\documentclass[10pt,twoside]{amsart}
\usepackage{amssymb, amsmath, mathrsfs, mathtools, amsfonts, amsthm, amscd, yfonts}
\usepackage{newunicodechar}
\usepackage{xcolor}
\usepackage{hyperref} 
\usepackage{multicol}

\theoremstyle{plain}
\newtheorem{Theorem}{Theorem}

\newtheorem{Proposition}[Theorem]{Proposition}
\newtheorem{Lemma}[Theorem]{Lemma}
\newtheorem{Corollary}[Theorem]{Corollary}

\newunicodechar{−}{-}
\theoremstyle{definition}
\newtheorem{Definition}[Theorem]{Definition}
\theoremstyle{remark}
\newtheorem{Remark}[Theorem]{Remark}
\newtheorem{Example}[Theorem]{Example}

\title[Regular gradings]{On infinite dimensional algebras with regular gradings}
\author[L. Centrone]{Lucio Centrone}
\address{Dipartimento di Matematica, Universit\`a degli Studi di Bari Aldo Moro, Via Edoardo Orabona, 4, 70125 Bari, Italy}
\email{lucio.centrone@uniba.it}
\thanks{L. Centrone was partially supported by PNRR-MUR PE0000023-NQSTI}
\author[P. Koshlukov]{Plamen Koshlukov}
\address{IMECC, UNICAMP, Rua S\'ergio Buarque de Holanda 651, 13083-859 Campinas, SP, Brazil}
\email{plamen@unicamp.br}
\thanks{P. Koshlukov was partially supported by FAPESP Grant 2024/14914−9, and CNPq Grant 307184/2023-4}
\author[K. Pereira]{Kau\^e Pereira}
\address{IMECC, UNICAMP, Rua S\'ergio Buarque de Holanda 651, 13083-859 Campinas, SP, Brazil}
\email{k200608@dac.unicamp.br}
\thanks{K. Pereira was supported by FAPESP Grant 2023/01673-0, and FAPESP Grant 2025/03763-2}


\begin{document}

\begin{abstract}
Let $G$ be a finite abelian group and let $K$ be an algebraically closed field of characteristic 0. We consider associative unital algebras $A$ over $K$ graded by $G$, that is $A=\oplus_{g\in G} A_g$, where the vector subspaces $A_g$ satisfy $A_gA_h\subseteq A_{g+h}$ for every $g$, $h\in G$. Such a $G$-grading is called regular whenever for every $n$-tuple $(g_1,\ldots,g_n)\in G^n$ there exist homogeneous elements $a_i\in A_{g_i}$ such that $a_1\cdots a_n\ne 0$ in $A$; furthermore, for every $g$, $h\in G$ and every $a_g\in A_g$, $a_h\in A_h$ one has $a_ga_h=\beta(g,h)a_ha_g$ for some $\beta(g,h)\in K^*$. Here $\beta(g,h)$ depends only on $g$ and $h$ but not on the elements $a_g$ and $a_h$. It is immediate that $\beta$ is a skew-symmetric bicharacter on $G$. The regular decomposition above is minimal if whenever $\beta(g,h)=\beta(g,k)$ for every $g\in G$, then $h=k$. In this paper we characterize the generators of the graded variety generated by the natural $\mathbb{Z}_{2}$-grading on the Grassmann algebra in terms of $\mathbb{Z}_{2}$-graded regular algebras with minimal regular decomposition. Furthermore we describe the finitely generated subalgebras of a $\mathbb{Z}_2$-graded regular algebra having a minimal regular decomposition. 

We recall that regular gradings and the corresponding decompositions play an important role in the description of numerical invariants of PI algebras as proved in the papers \cite{regevz2, bcommutation, bahturin2009graded, Eli1}.
\end{abstract}

\maketitle

\section*{Introduction}
Group gradings on algebras appeared long ago in Commutative Algebra: the polynomial algebras are naturally graded by the nonnegative integers. Moreover, finite dimensional simple Lie algebras are naturally graded by their roots. The notion of a Lie or Jordan superalgebra involves a grading by the cyclic group of order 2, $\mathbb{Z}_2$. In 1963 Wall \cite{wall} classified the finite dimensional simple associative $\mathbb{Z}_2$-graded algebras (here simple means that there are no nontrivial homogeneous ideals). Considering algebras with some additional structure, like grading, or involution, allows one to study suitable components of the algebra (homogeneous components in the case of gradings, or symmetric and skew elements in case of involution), and then reconstruct information about the whole algebra. 

In the theory of algebras with polynomial identities (PI theory), the systematic study of gradings and the corresponding graded identities started in the eighties. Kemer developed powerful tools that led him to the classification of the ideals of identities (also called T-ideals) of associative algebras in characteristic 0. As a consequence, he solved in the affirmative a long standing problem proposed by W. Specht in 1950: is every T-ideal in characteristic 0 finitely generated as a T-ideal? One of the main tools in Kemer's theory was the usage of $\mathbb{Z}_2$-gradings and the corresponding graded identities. For an account of Kemer's theory the reader can see for example \cite{kemer.1, kemer.2}, or the recent monograph \cite{Procesietalbook}. Kemer's results provided further motivation for studying group gradings and their graded identities. We refer the interested reader to the monograph \cite{elduque}. Although it concerns group gradings on Lie algebras it contains a good survey of the state-of-the-art in the associative case. Another very important tool in Kemer's theory is the Grassmann envelope of a $\mathbb{Z}_2$-graded algebra. If $A=A_0\oplus A_1$ is $\mathbb{Z}_2$-graded then its Grassmann envelope is $E(A)=A_0\otimes E_0\oplus A_1\otimes E_1$. Here $E$ stands for the Grassmann (or exterior) algebra of an infinite dimensional vector space $V$ with a basis $e_1$, $e_2$, \dots; then the elements 1 and $e_{i_1}e_{i_2}\cdots e_{i_k}$, $k\ge 1$, $i_1<i_2<\cdots<i_k$, form a basis of $E$. The multiplication in $E$ is defined by $e_ie_j=-e_je_i$ for all $i$ and $j$. Then $E_0$ is the span of 1 and all monomials as above, where $k$ is even; $E_1$ is the span of the monomials of odd length. Recall that $E_0$ is the centre of $E$, and $E_1$ is the "anti-commuting" part. We recall here why the Grassmann envelope is important. According to Kemer's theory, every PI algebra $A$ in characteristic 0 is PI equivalent to (that is  satisfies the same identities as) $E(B)$ for some finitely generated $\mathbb{Z}_2$-graded algebra $B$. Furthermore, as shown by Kemer, every finitely generated PI-algebra is PI-equivalent to a finite dimensional algebra. 
From now on we assume that the field $K$ is of characteristic 0. 
The notion of a regular grading was introduced by Regev and Seeman \cite{regevz2}. They considered a more general definition, which we recall here. Let $G$ be a finite group and let $A=\oplus_{g\in G}A_g$ be a $G$-graded algebra over the field $K$. We assume here $K$ is of characteristic 0. This grading is regular whenever two conditions are met. For every $n$ and for every $n$-tuple $(g_1,\ldots, g_n)$ of elements of $G$ there exist homogeneous elements $a_1\in A_{g_1}$, \dots, $a_n\in A_{g_n}$ such that $a_1\cdots a_n\ne 0$. Moreover, if $a_g\in A_g$, $a_h\in A_h$ then $a_ga_h=\beta(g,h) a_ha_g$. Here $\beta(g,h)\in K$ is a constant that depends only on $g$ and $h$ but not on the choice of $a_g$, $a_h$. The matrix $M=(\beta(g,h))$ of order $|G|$ is the matrix of the regular decomposition. It is immediate to see that the associativity of $A$ implies $\beta\colon G\times G\to K^{\ast}$ is a skew-symmetric bicharacter of $G$. Regev and Seeman proved that two $G$-graded algebras with regular decompositions are PI equivalent if and only if their matrices of the regular decompositions are conjugate via a permutation matrix. Later on, in \cite{bahturin2005finite, bahturin2009graded} the authors introduced the notion of a minimal regular decomposition. Informally speaking a regular decomposition is minimal whenever its matrix has no two identical columns. In the language of the $\beta(g,h)$ this is translated as follows. If for some $g$ and $h$ one has $\beta(g,t)=\beta(h,t)$ for all $t\in G$ then $g=h$. 

The matrix algebra $M_n(K)$ admits a natural $\mathbb{Z}_n\times \mathbb{Z}_n$-grading. It defines a minimal regular decomposition of $M_n(K)$. In \cite{bcommutation} the authors computed the matrix of this decomposition, its determinant equals $\pm n^{n^2}$. Hence in \cite{bahturin2009graded} the following conjecture was proposed: the decomposition matrix of a regular grading is invertible if and only if the grading is minimal. The authors of \cite{bcommutation} asked if the number of direct summands in the decomposition of a minimal regular grading, and also the determinant of the corresponding matrix, are invariants of the algebra (that is these do not depend on the minimal regular decomposition but on the algebra only). All these conjectures were confirmed by Aljadeff and David in \cite{Eli1}. In fact in \cite{Eli1} it was proved that the determinant of the matrix in question equals $\pm |G|^{|G|/2}$, and moreover that the PI exponent of the algebra equals $|G|$. Here we recall that the PI exponent of a PI algebra $A$ is one of the principal numerical invariants of the identities satisfied by $A$. We refer the reader to \cite{GZbook} for the definition of the codimension sequence of a PI algebra $A$, its PI-exponent and related topics.

In the present paper we continue the research initiated in \cite{LPK.1}. We assume that the base field $K$ is algebraically closed and of characteristic $0$, and we consider the group $G$ as the group $\mathbb{Z}_2$. We study minimal regular decompositions of associative unital algebras graded by $\mathbb{Z}_2$. The natural grading on the Grassmann algebra $E$ is a regular minimal one. Let $A$ admit a regular minimal $\mathbb{Z}_2$-grading. Then it turns out that $A$ and $E$ satisfy the same graded identities if and only if $\dim A=\infty$. These conditions are shown to be equivalent to $\beta(0,0)=\beta(0,1)=\beta(1,0)=1$, and $\beta(1,1)=-1$. The main objective of this paper is to study the variety of graded algebras defined by the natural $\mathbb{Z}_2$-grading on the Grassmann algebra in terms of the $\mathbb{Z}_{2}$-graded regular algebras having a minimal regular decomposition.  Given a $\mathbb{Z}_2$-graded algebra $A$ with a minimal regular decomposition, by Kemer's results we know that there is a  $\mathbb{Z}_2$-graded algebra $C$ such that the algebras $E(C)$, $A$, and $E$ satisfy the same graded identities. We prove that in our situation the finite dimensional algebra $C$ must be commutative. The structure of $C$ is described by using its Wedderburn--Malcev decomposition. We then describe the $\mathbb{Z}_2$-graded algebras $A$ that satisfy the same graded identities as $E$. If $A=E(C)$ as above, we prove that $A$ satisfies the same graded identities as $E$ if and only if $C$ is regular and commutative; another equivalent condition is obtained in terms of the Wedderburn--Malcev decomposition of $C$. Finally, we classify the finitely generated homogeneous subalgebras of a $\mathbb{Z}_2$-graded algebra $A$ with minimal regular decomposition and use this classification to show that every such algebra is a direct limit of ``almost'' regular algebras.
We think that our results will shed additional light on the behavior of the numerical invariants of PI-algebras and on the combinatorial relations among these invariants, in the context of regular gradings.

\section{Background}
Throughout we assume that $G$ is a finite abelian group, $X=\{x_{1}, x_2,\ldots\}$ is an infinite countable set of noncommuting variables and $K$ is an algebraically closed field of characteristic $0$. A $K$-algebra $A$ means an associative unital algebra over the field $K$, and we denote by 1 its unit.

\subsection{Basic definitions concerning graded polynomial identities}
Let $A$ be a $K$-algebra and let $\{A_g\mid g\in G\}$ be vector subspaces of $A$ such that $A=\oplus_{g\in G} A_g$. This decomposition of $A$ is a \textit{$G$-grading} whenever $A_gA_h\subseteq A_{gh}$ for every $g$, $h\in G$. The elements $0\ne a\in A_g$ are called \textsl{homogeneous} elements of degree $g$. Given a subspace $V\subseteq A$ we say $V$ is a \textit{graded subspace} (or \textit{homogeneous subspace}) if $V=\bigoplus_{g\in G} V\cap A_{g}$.  The notions of a graded subalgebra and a graded ideal of $A$ are defined in the natural way.

Split $X=\cup_{g\in G} X_g$ into a disjoint union of infinite sets, $X_g=\{x_1^{(g)}, x_2^{(g)}, \ldots\}$. We define a $G$-grading on the free associative algebra $K\langle X\rangle$ as follows. Declare the variables from $X_g$ homogeneous of degree $g$, and then extend this to the monomials in $K\langle X\rangle$ in the natural way.
We denote the resulting $G$-graded algebra by $K\langle X_{G}\rangle$; it is the \textit{free} $G$-graded algebra. A polynomial 
\[
f=f(x_{1}^{(g_{1})},\ldots, x_{k_{1}}^{(g_{1})},x_{1}^{(g_{2})},\ldots, x_{k_{2}}^{(g_{2})},\ldots, x_{1}^{(g_{m})},\ldots, x_{k_{m}}^{(g_{m})})\in K\langle X_{G}\rangle
\]
is a \textsl{graded polynomial identity} for the $G$-graded algebra $R$,  if for every homogeneous elements $r_{i}^{(g_{i})}$, \dots, $r_{k_{i}}^{(g_{i})}\in R_{g_{i}}$, $1\leq i\leq m$,  we have
\[
f(r_{1}^{(g_{1})},\ldots, r_{k_{1}}^{(g_{1})},r_{1}^{(g_{2})},\ldots, r_{k_{2}}^{(g_{2})},\ldots, r_{1}^{(g_{m})},\ldots, r_{k_{m}}^{(g_{m})})=0.
\]    
The set of all graded identities for $R$ is denoted by $T_{G}(R)$. It can be shown that $T_{G}(R)$ is a graded ideal of $K\langle X_{G}\rangle$, moreover it is stable under every graded endomorphism of $K\langle X_{G}\rangle$. It is called the $T_{G}$-ideal of $R$. 

\begin{Definition} Let $R$ be a $G$-graded algebra. The graded variety generated by $R$, denoted by $\operatorname{var}^{G}(R)$, is the class of all $G$-graded algebras $B$ such that $T_{G}(R)\subseteq T_{G}(B)$. 

Moreover, if $U\in \operatorname{var}^{G}(R)$ is such that $T_{G}(U)=T_{G}(R)$, then we say that $U$ generates the variety $\operatorname{var}^{G}(R)$.
\end{Definition}
For a more general definition of a variety of algebras, definition of the ordinary PI-exponent, the graded PI-exponent, as well as for fundamental theorems in the theory of PI-algebras and graded PI-algebras, we refer the reader to the monographs \cite{drensky2000free} and \cite{GZbook}. In the next section we will use the notions of direct systems and direct limits; we refer the reader to  \cite{rotman2009introduction} and \cite{weibel1994introduction} for further information about these notions.

\subsection{Basic notions on regular gradings}
 
 We start with the main definition (see \cite{bahturin2009graded}). 

\begin{Definition}
\label{def.regular.gradings}
Let $A$ be a $G$-graded algebra. Then $A$ is called a $G$-graded \emph{regular} algebra (or $A$ has a regular grading with respect to $G$), with regular decomposition given by the $G$-grading, if the following conditions hold:
\begin{enumerate}
    \item[(i)] Given $n\in \mathbb{N}$ and an $n$-tuple $(g_{1},\ldots, g_{n})\in G^{n}$, there exist homogeneous elements $a_{1}\in A_{g_{1}}$, \dots, $a_{n}\in A_{g_{n}}$ such that $a_{1}\cdots a_{n}\neq 0$.
     \item[(ii)] For every $g$, $h\in G$ and for every $a_{g}\in A_{g}$, $a_{h}\in A_{h}$, there exists $\beta(g,h)\in K^{\ast}$ satisfying
     \[
     a_{g}a_{h}=\beta(g,h)a_{h}a_{g}.
     \]
\end{enumerate}
Moreover, we define the \emph{regular decomposition matrix} associated with the regular decomposition
of the $G$-graded algebra $A$ as $M^{A}=(\beta(g,h))_{g,h}$.
\end{Definition}
Note that $M^A$ is a square matrix of order $|G|$. 

\begin{Remark} Given a $G$-graded algebra $A$, we define the support of $A$ by
\[
\operatorname{Supp}_{G}(A)=\{g\in G\mid A_{g}\neq 0\}.
\]
We observe that if $A$ is a $G$-graded regular algebra, then by the second condition of regularity we have $\operatorname{Supp}_{G}(A)=G$.
    
\end{Remark}

Since the product in $A$ is associative, then the function $\beta\colon G\times G\rightarrow K^{\ast}$ must be a skew-symmetric bicharacter in the following sense: for every $g$, $h$ and $s\in G$ we have
\[
\beta(g,h)=\beta(h,g)^{-1},\quad \beta(g+s,h)=\beta(g,h)\beta(s,h),\quad \text{and}\quad\beta(g,s+h)=\beta(g,s)\beta(g,h).
\]
Recall that a skew-symmetric bicharacter $\theta\colon G\times G\rightarrow K^{\ast}$ is trivial, if $\theta(g,h)=1$, for every $g$, $h\in G$. 
\begin{Example} The following are examples of regular gradings.
\begin{enumerate}
    \item  The commutative polynomial algebra $K[x]$ becomes an infinite dimensional  $\mathbb{Z}_n$-graded regular algebra, for $n \in \mathbb{N}$, with the trivial skew-symmetric bicharacter, by setting
\[
K[x] = \oplus_{0 \leq i \leq n-1} x^i K[x^n]
\]
(see \cite[Example 5]{LPK.1}).
\item Given a cocycle $\alpha\in H^{2}(G,K^{\ast})$, the twisted group algebra $K^{\alpha}G$ is a finite dimensional $G$-graded regular algebra with skew-symmetric bicharacter given by $\beta(g,h)=\alpha(g,h)\alpha(h,g)^{-1}$, for every $g$, $h\in G$ (see \cite[Example 8]{LPK.1}).

\item The infinite dimensional  Grassmann algebra $E$, with its natural $\mathbb{Z}_{2}$-grading $E=E_{0}\oplus E_{1}$, is an infinite dimensional  $\mathbb{Z}_{2}$-graded regular algebra with skew-symmetric bicharacter given by 
\[
\beta(0,0)=\beta(0,1)=\beta(1,0)=1,\quad \text{and}\quad \beta(1,1)=-1
\]
(see \cite[Example 6]{LPK.1}).
\end{enumerate}
\end{Example}
For further examples see \cite{LPK.1}.

\begin{Definition} Let $A$ be a $G$-graded regular algebra with skew-symmetric bicharacter $\beta\colon G\times G\rightarrow K^{\ast}$. The regular decomposition of $A$ is \textsl{nonminimal} if there exist $h$, $k\in G$ with $h\neq k$ such that $\beta(g,h)=\beta(g,k)$, for every $g\in G$. Otherwise we say the regular decomposition of $A$ is minimal.
\end{Definition}

Bahturin and Regev conjectured in \cite[Conjecture 2.5]{bahturin2009graded} that if $A$ is a $G$-graded regular algebra, then the regular decomposition of $A$ is minimal if and only if $\det M^{A} \neq 0$. Aljadeff and David, in \cite[Theorem 7]{Eli1}, gave a positive answer to this conjecture when the ground field is algebraically closed of characteristic $0$. More recently, in \cite{LPK.1}, we proved that the conjecture raised by Bahturin and Regev fails over fields of characteristic $p>2$. 

Since $K$ is an algebraically closed field of characteristic $0$, \cite[Theorem 7]{Eli1} gives us a powerful tool for determining whether the regular decomposition of a regular grading is minimal. For instance, the regular decomposition of the Grassmann algebra $E$ with its natural $\mathbb{Z}_{2}$-grading is minimal, because
\[
\det M^{E}=\det\begin{pmatrix}
    1 & 1\\
    1 & -1
\end{pmatrix}=-2\neq 0.
\]
It follows from \cite[Theorem 7]{Eli1} that if $B$ is a $G$-graded regular algebra with a minimal regular decomposition, then $\exp(B) = |G|$. Here $\exp(B)$ stands for the PI exponent of the PI algebra $B$, see for example \cite{GZbook}.

Finite dimensional regular gradings are closely related to twisted group algebras. Namely, it was shown in \cite[Theorem 48]{LPK.2} that if $A$ is a finite dimensional $G$-graded regular algebra such that the product of all its graded simple components with $J(A)$ is nonzero, then there exist a commutative algebra $W(A)$, a $2$-cocycle $\alpha \in H^2(G,K^\ast)$, and a $G$-graded algebra $\mathscr{W}$, which is not regular, such that $A$ is graded isomorphic to $(K^\alpha G \otimes W(A))\oplus \mathscr{W}$, and $J(W(A))\oplus J(\mathscr{W})_{0}=J(A_{0})$. Here, $W(A)$ is considered with the trivial grading, and the grading on $K^\alpha G \otimes W(A)$ is given by $(K^\alpha G \otimes W(A))_g = (K^\alpha G)_g \otimes W(A)$, for every $g \in G$. We note that the general case, without the assumption that the product of all graded simple components with the Jacobson radical of $A$ is nonzero, was also considered in \cite[Theorem 50]{LPK.2}. Here, $J(R)$ denotes the Jacobson radical of an algebra $R$.

As a consequence of this result, it follows that in the finite dimensional case the converse of \cite[Theorem 7]{Eli1} is valid, thereby yielding the following characterization.

\begin{Theorem}\cite[Theorem 50]{LPK.2} Let $A$ be a finite dimensional $G$-graded regular algebra with skew-symmetric bicharacter $\beta$. Then the regular decomposition of $A$ is minimal if and only if $\exp(A)=|G|$. Equivalently, $\det M^{A}\neq 0$ if and only if $\exp(A)=|G|$.

\end{Theorem}

\section{\texorpdfstring{On the structure of infinite dimensional  $\mathbb{Z}_{2}$-graded regular algebras}{On structure of infinite dimensional  Z{2}-graded regular algebras}}

\subsection{On certain graded subalgebras of regular algebras}

The next lemma provides a necessary condition for a graded algebra with a regular grading to be finite dimensional.
\begin{Lemma}\label{infinite} Let $A$ be a $G$-graded regular algebra with a bicharacter $\beta$. If there exists $h\in G$ such that $\beta(h,h)=-1$, then $\dim (A)=\infty$.
\end{Lemma}
\begin{proof}
    Suppose $\dim A<\infty$. It follows from \cite[Proposition 26]{LPK.2} that there exists a graded subalgebra $D\subseteq A$ such that $D$ is graded isomorphic to  the twisted group algebra $K^{\sigma}G$, where $\sigma\in H^{2}(G,K^{\ast})$ is a 2-cocycle. Notice that if $0\neq X_{h}\in D_{h}$, then $X_{h}^{2}=X_{h}X_{h}=-X_{h}X_{h}=-X_{h}^{2}$, and thus $X_{h}^{2}=0$. It follows that
    \[
    0=X_{h}^{2}=X_{h}X_{h}=\sigma(h,h)X_{2h}\neq 0
    \]
    which is a contradiction. 
\end{proof}

The converse of the lemma is not always true, as shown in the following example.

\begin{Example}  
Consider the polynomial algebra in two commutative variables $K[x,t]$. Let $I$ be the ideal of $K[x,t]$ generated by the polynomial $t^{2}-1$. In the algebra $A := K[x,t]/I$, we write a coset $\overline{a}$ simply as $a$. Thus, $A$ is $\mathbb{Z}_{2}$-graded by setting  
\[
A_{0} := K[x],\quad A_{1} := tK[x].
\]  
Since $t^{2} = 1$ in $A$, it follows that $A$ is a $\mathbb{Z}_{2}$-graded regular algebra with a trivial skew-symmetric bicharacter (because it is a commutative algebra) and $\dim A = \infty$.  
\end{Example}

Now, let $A$ be a $\mathbb{Z}_{2}$-graded regular algebra with a skew-symmetric bicharacter $\beta$, whose regular decomposition is minimal. Observe that \[ \beta(0,1) = \beta(0+0,1) = \beta(0,1) \beta(0,1), \]
which implies that $\beta(0,1) = 1$. Similarly, it can be verified that $\beta(0,0) = \beta(1,0) = \beta(0,1) = 1$, which in turn implies that  the minimality of the decomposition determines $\beta(1,1)$. Since $\beta(1,1) \in \{\pm 1\}$, and $\det M^{A}\neq 0$, we conclude that $\beta(1,1) = -1$. In particular, by Lemma \ref{infinite} we have $\dim A=\infty$. Additionally, since $M^{A} = M^{E}$, it follows from \cite[Lemma 31]{Eli1} that $T_{\mathbb{Z}_{2}}(A) = T_{\mathbb{Z}_{2}}(E)$.  
Such an analysis shows that if the regular decomposition of a $\mathbb{Z}_{2}$-graded algebra is minimal, then its skew-symmetric bicharacter is automatically determined. Moreover, this skew-symmetric bicharacter is precisely the one associated with the regular $\mathbb{Z}_2$-grading of the Grassmann algebra $E$ (with respect to its natural grading). Furthermore, $\dim A=\infty$ and $A$ satisfies the same graded identities as $E$. Recall that the latter were completely described in \cite{giambruno2001polynomial}. We summarize all these remarks in the following.

\begin{Lemma}  
\label{remark.superalgebra}
Let $A$ be a $\mathbb{Z}_{2}$-graded algebra with a regular grading determined by the skew-symmetric bicharacter $\beta$, whose corresponding regular decomposition is minimal. Then, the following statements hold:
\begin{enumerate}
    \item $\beta$ is such that
    \[
    \beta(0,0)=\beta(0,1)=\beta(1,0)=1,\quad \text{and}\quad \beta(1,1)=-1,
    \]
    \item $\dim A=\infty$,
    \item  $T_{\mathbb{Z}_{2}}(A)=T_{\mathbb{Z}_{2}}(E)$.
    \end{enumerate}
\end{Lemma}

We observe that the result of Lemma \ref{remark.superalgebra} is stronger than the converse of Lemma \ref{infinite} when assuming that the regular decomposition is minimal.

Before proceeding, we need a weaker definition of regularity. Recall that if $R$ is $G$-graded then $R$ is $\beta$-commutative whenever for every $g$, $h\in G$ and for every $r_g\in R_g$ and $r_h\in R_h$ one has $r_gr_h=\beta(g,h) r_hr_g$ for some scalars $\beta(g,h)$ that depend only on $g$ and $h$ but not on the choice of the elements of $R_g$ and $R_h$.
\begin{Definition}
    Given a $G$-graded $\beta$-commutative algebra $R$, we say that $R$ is an \emph{$m$-regular} algebra if, for every $k \leq m$ and any $k$-tuple $(g_{1},\ldots,g_{k}) \in G^{k}$, there exist elements $r_{g_{i}} \in R_{g_{i}}$ such that  
\[
r_{g_{1}} \cdots r_{g_{k}} \neq 0.
\]
\end{Definition}
\begin{Example}
\label{Exemplo.weakly.regular}
    \begin{enumerate}
        \item Every $G$-graded regular algebra is $k$-regular for each $k\in \mathbb{N}$.
        \item For every $n\in \mathbb{N}$, let $V$ be a vector space with an ordered basis $e_{1}$, \dots, $e_{n}\in V$. Recall that the Grassmann algebra on $n$ generators (or of a vector space of dimension $n$) constructed over $V$ and denoted by $E_{n}$, is the associative algebra generated by $e_{1}$, \dots, $e_{n}$, with defining relations
        \[
        e_{i}e_{j}+e_{j}e_{i}=0,\quad \text{for}\quad 1\leq i,j\leq n.
        \]
        It can be shown that the set 
        \[
        \mathscr{T}:=\{e_{i_{1}}\cdots e_{i_{k}}\mid 1\leq i_{1}<\cdots < i_{k}\leq n,\quad \text{for} \quad k\leq n\},
        \]
        together with $1\in K$, forms a basis of $E_{n}$, thus $\dim E_n=2^n$. We notice that $E_{n}$ has a natural $\mathbb{Z}_{2}$-grading given by
        \begin{align*}
            (E_{n})_{0}&=\text{span}_{K}\{1\}\cup \{e_{i_{1}}\cdots e_{i_{k}}\in \mathscr{T}\mid \text{$k$ is even}\}\\
            (E_{n})_{1}&=\text{span}_{K}\{e_{i_{1}}\cdots e_{i_{k}}\in \mathscr{T}\mid \text{$k$ is odd}\}.
        \end{align*}
       Then, by definition it is clear that $E_{n}$ is $n$-regular, but is not $n+1$-regular.
    \end{enumerate}
\end{Example}

\begin{Lemma}
\label{Grassmann.finite}
Let $A$ be a regular $\mathbb{Z}_{2}$-graded algebra with skew-symmetric bicharacter $\beta$ whose regular decomposition is minimal. Then, for every $n\in \mathbb{N}$, there exists a finite dimensional $\mathbb{Z}_{2}$-graded algebra $F_{n}\subseteq A$, which is $n$-regular and isomorphic to $E_{n}$, the Grassmann algebra of a vector space of dimension $n$.
\end{Lemma}
\begin{proof}Given $n\in \mathbb{N}$, let  
$\underbrace{(1,\ldots,1)}_{n}\in \mathbb{Z}_{2}^{n}$.  
By regularity, there exist $a_{1}$, \dots, $a_{n} \in A_{1}$ such that $a_{1} \cdots a_{n} \neq 0$. For every $k \leq n$, we denote by $\mathscr{U}_{k}$ the set of monomials of length $\le k$ on the elements $a_1$, \dots, $a_n$. Hence, $\mathscr{U}_{k}$ consists of the products of the form $a_{i_{1}} \cdots a_{i_{m}}$, where $1 \leq i_{1} < \cdots < i_{m} \leq n$, and $m\le k$.  

Consider the subset $S$ of $\mathbb{N} \cup \{0\}$ consisting of all $q \in \mathbb{N} \cup \{0\}$ such that there exist $f_{1}$, \dots, $f_q\in \mathscr{U}_{k}$, for some $k\le n$, and nonzero scalars $c_{1}$, \dots, $c_{q}$ such that  
\begin{equation}  
\label{eq.1}  
c_{1}f_{1} + \cdots + c_{q}f_{q} = 0.  
\end{equation}  

If $S$ is nonempty, then we choose the least possible $q > 0$, along with $f_{1}$, \dots, $f_{q} \in \mathscr{U}_{k}$, for the appropriate $k$, and $c_{1}$, \dots, $c_{q} \in K \setminus \{0\}$ satisfying (\ref{eq.1}). 

If all $f_i$ correspond to the same product of some elements out of $a_{j_1}$, \dots, $a_{j_p}$, then $q=1$ and (\ref{eq.1}) becomes $c(a_{j_1}\cdots a_{j_p})=0$ for some $0\ne c\in K$, a contradiction. On the other hand, if some $a_r$ appears in some, but not all of the $f_i$ then we multiply (\ref{eq.1}) by this $a_r$. Thus all summands that contain $a_r$ vanish, and we are left with a shorter linear combination. This contradicts the minimality of $q$. Therefore the sets $\mathscr{U}_{k}$ are linearly independent. 

We denote by $F_{n}$ the vector space with basis $\mathscr{U}:=\{1\}\cup \mathscr{U}_{1}\cup \cdots \cup \mathscr{U}_{n}$. By construction, since the generators of $F_n$ are homogeneous in the grading on $A$, it follows that $F_{n}$ is a graded subalgebra of $A$. Now we take $E_{n}$, the Grassmann algebra on $n$ generators $e_{1}$, \dots, $e_{n}$, and the basis $\mathscr{T}$ described in Example \ref{Exemplo.weakly.regular}. We shall construct an explicit (and at this stage quite obvious) isomorphism between $F_{n}$ and $E_{n}$. Consider the unique linear map $\nu\colon F_{n}\to E_{n}$ defined by mapping each  $a_{i_{1}}\cdots a_{i_{k}}\in \mathscr{U}$ to $e_{i_{1}}\cdots e_{i_{k}}\in \mathscr{T}$. It is well defined because $\mathscr{U}$ is a basis of $F_n$.

Since $\nu$ carries a basis of $F_{n}$ into a basis of $E_{n}$, it follows that $\nu$ is an isomorphism of graded vector spaces. To show that $\nu$ is an algebra homomorphism, it is enough to verify the homomorphism condition on the basic elements. Take $a=a_{i_{1}}\cdots a_{i_{m}}$ and $b=a_{j_{1}}\cdots a_{j_{l}}\in \mathscr{U}$, and notice that if $\{i_{1},\ldots, i_{m}\}\cap \{j_{1},\ldots, j_{l}\}\neq \emptyset$, then
\[
\nu(ab)=0=\nu(a)\nu(b).
\]
Suppose $\{i_{1},\ldots, i_{m}\}\cap \{j_{1},\ldots, j_{l}\}=
\emptyset$. We can order the elements $i_{1},\dots, i_{m}, j_{1},\dots, j_{l}$ to obtain $d_{1}<\cdots <d_{k+l}$. This reordering might change only the sign of the corresponding element in $F_n$. But by the relations of $F_{n}$, it follows that there exists $N\in \mathbb{N}$ such that
\[
\nu(ab)=\nu((-1)^{N}a_{d_{1}}\cdots a_{d_{k+l}})=(-1)^{N}e_{d_{1}}\cdots e_{d_{k+l}}.
\]
Here $N$ counts how many transpositions one has to make in order to obtain the correct ordering of the subscripts (of course, one takes this number modulo 2 afterwards). In effect, $N$ is the usual sign of the permutation of the indices $(d_1, d_2, \ldots, d_{k+l})$. 
On the other hand, by the multiplication rules in $E_n$ we will have the same $N$ when reordering the subscripts, that is
\[
\nu(a)\nu(b)=e_{i_{1}}\cdots e_{i_{k}}e_{j_{1}}\cdots e_{j_{l}}=(-1)^{N}e_{d_{1}}\cdots e_{d_{k+l}}.
\]
Therefore $\nu(ab)=\nu(a)\nu(b)$. Hence, $\nu\colon F_{n}\rightarrow E_{n}$ is the desired isomorphism of graded algebras, and we are done.
\end{proof}
\textbf{Notation:}  Let $R$ be an algebra with unity and suppose $R$ admits a grading by the nonnegative integers, $R=\oplus_{i\ge 0} R_i$. Then $R_iR_j\subseteq R_{i+j}$, and if we set $R^{+}=\sum_{i\ge 1} R_i$ then we get that $R^+$ is a subalgebra of $R$, and $1\notin R^+$. We shall apply this to the Grassmann algebras (finite or infinite dimensional ).

{ 
Although a regular algebra contains a finite dimensional subalgebra isomorphic to $E_{n}$, the Grassmann algebra on $n$ generators, for every $n\in \mathbb{N}$, it may happen that a $\mathbb{Z}_{2}$-graded regular algebra does not contain a copy of the infinite dimensional  Grassmann algebra. The following example illustrates this fact.

\begin{Example} 
\label{example_referee}
Consider a set of non-commutative variables $Y=\{x_{i,n}\mid i\leq n,\; n\in \mathbb{N}\}$ and let $A$ be the unital associative algebra defined by the generators $Y$ subject to the relations
\[
x_{i,n}x_{j,m}=-\delta_{n,m}x_{j,m}x_{i,n},\quad i\leq n,\quad j\leq m,\quad n,m\in \mathbb{N},
\]
where $\delta_{n,m}$ denotes the Kronecker delta. Then $A$ can be identified with the algebra
\[
K\oplus\Big(\bigoplus_{n\in \mathbb{N}}E_{n}^{+}\Big),
\]
where $E_{n}^{+}$ is the finite dimensional non-unital Grassmann algebra generated by $\{x_{i,n}\mid i\leq n\}$. 

Thus, $A$ has a natural structure of a $\mathbb{Z}_{2}$-graded regular algebra given by
\[
A_{0}=K\oplus \Big(\bigoplus_{n\in \mathbb{N}} (E_{n})_{0}\Big),\quad \text{and}\quad A_{1}=\bigoplus_{n\in \mathbb{N}} (E_{n})_{1}.
\]
Of course, the regular decomposition of $A$ is minimal because $M^{A}=M^{E}$. Suppose that there exists a graded subalgebra $B$ of $A$ such that $B \cong E$. Let $e_{1}, e_{2}, \ldots, e_{n}, \ldots$ be the generators of the Grassmann algebra $E$. Then there exists a graded embedding $\psi\colon E\rightarrow A$.

Thus, $\psi(e_{1})=c_{1}+\cdots+c_{k}$, where $c_{i}\in (E_{i})_{1}$, $1\leq i\leq k$. Now, for $2\leq \ell\leq k+1$, we can write $\psi(e_{\ell})=a_{\ell}+b_{\ell}$, where $a_{\ell}\in \bigoplus_{i=1}^{k}(E_{i})_{1}$ and $b_{\ell}\in \bigoplus_{j\geq k+1} (E_{j})_{1}$. By construction we have $\psi(e_{1})b_{\ell}=0$ for all $2\leq \ell\leq k+1$. Then
\[
    \psi(e_{1}\cdots e_{k+1}) 
    = \psi(e_{1})\cdots \psi(e_{k+1})
    = \psi(e_{1})a_{2}\cdots a_{k+1}\in \bigoplus_{i=1}^{k}(E_{i}^{+})^{k+1}.
\]
Therefore, since $(E_{n}^{+})^{m}=0$ if $m>n$, we obtain $\psi(e_{1}\cdots e_{k+1})=0$, that is, $e_{1}\cdots e_{k+1}=0$, which is a contradiction because this is a nonzero basic element of $E$. We conclude that there does not exist a graded subalgebra of $A$ isomorphic to $E$.
\end{Example}}

\begin{Definition}
\label{Defi.strongly}
Let $A$ be a $\mathbb{Z}_{2}$-graded algebra with a minimal regular grading. We say the grading on $A$ is \textsl{strongly regular} if for every homogeneous element $0\neq b\in A$, there exists $c\in A_{1}$ such that $bc\neq 0$.
\end{Definition}

\begin{Proposition}
\label{strongly.regular}
Let $A$ be a $\mathbb{Z}_{2}$-graded algebra with a minimal regular grading that is strongly regular as well, with associated bicharacter $\beta$. Then there exists $B$, a $\mathbb{Z}_2$-graded subalgebra of $A$, that is isomorphic to the infinite dimensional  Grassmann algebra $E$ with its canonical grading.
\end{Proposition}
\begin{proof}
    Given an arbitrary $0\neq a_{1}\in A_{1}$, we choose $a_{2}\in A_{1}$ such that $a_{1}a_{2}\neq 0$, by the strong regularity. Fix a positive integer $n>1$, suppose we have constructed $a_{3}$, \dots, $a_{n}\in A_{1}$ such that $a_{1}a_2 \cdots a_{n}\neq 0$. Then, once again by the strong regularity, there exists $a_{n+1}\in A_{1}$ such that $a_{1}\cdots a_{n}a_{n+1}\neq 0$. On the other hand, by the proof of Lemma \ref{Grassmann.finite} we conclude that $\Omega=\{a_{1}\cdots a_{n}\mid n\in \mathbb{N}\}$ is a linearly independent set. It follows that the algebra $B$ generated by $\Omega$ is isomorphic to $E$, as $\mathbb{Z}_2$-graded algebras.
\end{proof}

\subsection{ Constructing regular gradings via direct limits}

\begin{Lemma}
\label{graded.algebra}
Let $\{A_{n} \mid n\in \mathbb{N}\}$ be a family of $G$-graded algebras and suppose that for every $n\leq m$, there exists a graded homomorphism $\phi_{m}^{n}\colon A_{n}\to A_{m}$, such that $(A_{n},\phi_{m}^{n})_{n\leq m}$ is a direct system. Then  the direct limit $\varinjlim A_{n}$ is a $G$-graded algebra. 
\end{Lemma}

\begin{proof} Denote by $A:= \varinjlim A_{n}$ the direct limit. Let us consider the canonical injections $\lambda_{i} \colon A_{i} \to \bigoplus_{i} A_{i}$, for $i\in \mathbb{N}$.

It is well known that
\[
\varinjlim A_{n} = \Bigg(\bigoplus_{n} A_{n}\Bigg) / S,
\]
where $S$ is the vector space spanned by the elements $\lambda_{j} \phi_{j}^{i}(a_{i}) - \lambda_{i}(a_{i})$, for $i\leq j$. By \cite[Lemma 5.30]{rotman2009introduction}, we have that every element $x\in A$ can be written as $x=\varepsilon_{i}(a_{i})$, where the linear map $\varepsilon_{i}\colon b\in A_{i}\mapsto \lambda_{i}(b)+S\in A$ satisfies
\[
\varepsilon_{j}\phi_{j}^{i}=\varepsilon_{i},\quad \text{for}\quad i\leq j.
\]
Given two elements $x=\varepsilon_{i}(a_{i})$ and $y=\varepsilon_{j}(a_{j})$, take $k\geq \max\{i,j\}$. Since $\phi_{k}^{i}(a_{i})\phi_{k}^{j}(a_{j})\in A_{k}$, we define a product on $A$ in the following way:
\begin{equation}
\label{eq.prod.}
    x\cdot y:= \varepsilon_{k}(\phi_{k}^{i}(a_{i})\phi_{k}^{j}(a_{j})).
\end{equation}
Thus (\ref{eq.prod.}) defines an algebra structure on $A$.

Before proceeding with the proof, for every $g \in G$ and $n \leq m$, we define the linear map $(\varphi_{g})_{m}^{n} := (\phi_{m}^{n})_{\mid (A_{n})_{g}}$. It is well-defined because $\phi_{m}^{n}$ is a graded homomorphism, and therefore $(\varphi_{g})_{m}^{n}$ maps $(A_{n})_{g}$ into $(A_{m})_{g}$. Additionally, if $p \leq n \leq m$, since $(A_{n}, \phi_{m}^{n})$ is a direct system, we obtain
\[
(\varphi_{g})_{m}^{n} (\varphi_{g})_{n}^{p} = (\phi_{m}^{n} \phi_{n}^{p})_{\mid (A_{p})_{g}} = (\phi_{m}^{p})_{\mid (A_{p})_{g}}.
\]
Hence, $((A_{n})_{g}, (\varphi_{g})_{m}^{n})_{n \leq m}$ is a direct system.

On the other hand, since the direct limit commutes with direct sums we have
\[
A=\varinjlim A_{n}=\varinjlim \bigoplus_{g\in G} (A_{n})_{g}=\bigoplus_{g\in G}\varinjlim (A_{n})_{g},
\]
and this implies that $A$ is $G$-graded as a vector space, with
\[
(\varinjlim A_{n})_{g}=\varinjlim (A_{n})_{g},\quad \text{for every}\quad g\in G.
\]
 We have to show that $A$ is a graded algebra. To this end, we choose arbitrary $g$, $h\in G$, and take $x=\varepsilon_{i}(a)\in \varinjlim (A_{n})_{g}$, and $y=\varepsilon_{j}(b)\in \varinjlim (A_{n})_{h}$. Here $a\in (A_{i})_{g}$ and $b\in (A_{j})_{h}$ for some $i$, $j\in \mathbb{N}$. By hypothesis the maps $\{\phi_{m}^{n}\mid n\leq m\}$ are graded homomorphisms, and hence, if $k\geq \max\{i,j\}$, then $\phi_{k}^{i}(a)\in (A_{k})_{g}$ and $\phi_{k}^{j}(b)\in (A_{k})_{h}$. All this implies $\phi_{k}^{i}(a)\phi_{k}^{j}(b)\in (A_{k})_{g+h}$. Therefore, we have
\[
x\cdot y=\varepsilon_{k}(\phi_{k}^{i}(a)\phi_{k}^{j}(b))\in \varinjlim (A_{n})_{g+h}=(\varinjlim A_{n})_{g+h}.
\]
In other words, $A$ is a $G$-graded algebra. 
\end{proof}

We observe here that for the product defined above, if $x=\varepsilon_{i}(a_{i})$,  $y=\varepsilon_{i}(b_{i})$, $a_{i}$, and $b_{i}\in A_{i}$, then
\[
x\cdot y= \varepsilon_{i}(a_{i}b_{i}).
\]

The following proposition gives a method to construct regular gradings starting from $m$-regular ones.
\begin{Proposition}  
\label{family.regular}
Let $\{A_{n} \mid n\in \mathbb{N}\}$ be a family of $G$-graded $\beta$-commutative algebras, satisfying the following conditions:
\begin{enumerate}  
    \item for every $n\leq m$, there exists an injective graded homomorphism  $\phi_{m}^{n}\colon A_{n}\rightarrow A_{m}$, such that $(A_{n},\phi_{m}^{n})$ forms a direct system; 
    \item for every $n\in \mathbb{N}$, the algebra $A_n$ is $n$-regular.  
\end{enumerate}  
Then the direct limit $A = \varinjlim A_{n}$ is a $G$-graded regular algebra with skew-symmetric bicharacter $\beta$.  
\end{Proposition}
\begin{proof} We shall use the same notations as in Lemma \ref{graded.algebra}. By that lemma we know that $A$ is a $G$-graded algebra and
\[
A=\bigoplus_{g\in G}\varinjlim (A_{n})_{g}.
\]

Given $g$, $h \in G$, let $a_{g} \in \varinjlim (A_{n})_{g}$ and $a_{h} \in \varinjlim (A_{n})_{h}$. By \cite[Lemma 5.30]{rotman2009introduction}, and Lemma \ref{graded.algebra} we can write $a_{g} =\varepsilon_{i}(a_{i}^{g})$ and $a_{h} =\varepsilon_{j}(a_{j}^{h})$, where $a_{i}^{g} \in (A_{i})_{g}$ and $a_{j}^{h} \in (A_{j})_{h}$. If $k\geq \max\{i,j\}$, it follows that
\begin{align*}
a_{g}\cdot a_{h} &= \varepsilon_{k}(\phi_{k}^{i}(a_{i}^{g})\phi_{k}^{j}(a_{j}^{h}))\\
& = \beta(g,h)\varepsilon_{k}(\phi_{k}^{j}(a_{j}^{h})\phi_{k}^{i}(a_{i}^{g}))\\
 &= \beta(g,h)a_{h}\cdot a_{g},
\end{align*}
and this implies that $A$ is a $G$-graded $\beta$-commutative algebra. 

Now we  show $A$ satisfies the first condition of regularity. Let $k \in \mathbb{N}$ and $(g_{1},\ldots, g_{k})\in G^{k}$. By the $k$-regularity of $A_{k}$, there exist elements $a_{k}^{g_{1}} \in (A_{k})_{g_{1}}$, \dots, $a_{k}^{g_{k}} \in (A_{k})_{g_{k}}$ such that  
\[
a_{k}^{g_{1}} \cdots a_{k}^{g_{k}} \neq 0.
\]  
Let $a := a_{k}^{g_{1}} \cdots a_{k}^{g_{k}}\in A_{k}$ and $\widetilde{g} = g_{1} + \cdots + g_{k}$, and consider  
\[
\overline{a} = \varepsilon_{k}(a) \in \varinjlim (A_{n})_{\widetilde{g}}.
\]
Suppose that $\overline{a} = 0$, then, by \cite[Lemma 5.30]{rotman2009introduction}, there exists $t\geq k$ such that $\phi_{t}^{k}(a) = 0$. On the other hand, since $\phi_{t}^{k}$ is injective, it follows that $a = 0$ which is a contradiction. We conclude that $A$ is a $G$-graded regular algebra with skew-symmetric bicharacter $\beta$.
\end{proof}

\begin{Example} Let $E_{n}$ be the Grassmann algebra generated by $n$ elements. Since for every $n\leq m$ we have  $E_{n}\subseteq E_{m}$, it follows that
\[
\varinjlim E_{n}= \bigcup_{n\in \mathbb{N}}E_{n}=E.
\]
\end{Example}

At the end of this paper, we show that, in fact, all infinite dimensional  $\mathbb{Z}_{2}$-graded regular gradings with minimal regular decomposition can be realized as direct limits (with respect to a directed set) of $m$-regular gradings.

 \section{A characterization of the graded variety generated by the Grassmann algebra}
 
Throughout this section, we let $A$ denote a $\mathbb{Z}_{2}$-graded regular algebra, and suppose that the regular decomposition of $A$ is minimal. Therefore, by Lemma \ref{remark.superalgebra} the skew-symmetric bicharacter $\beta\colon \mathbb{Z}_{2}\times \mathbb{Z}_{2}\rightarrow K^{\ast}$ of $A$ is given by
\[
\beta(0,0)=\beta(0,1)=\beta(1,0)=1,\quad \text{and}\quad \beta(1,1)=-1.
\]

Let $A=A_0\oplus A_1$ be a $\mathbb{Z}_2$-graded algebra, we denote by $E(A)=A_0\otimes E_0\oplus A_1\otimes E_1$ its Grassmann envelope. This construction is canonical in the definition of superalgebras of a variety of (not necessarily associative) algebras. While in the associative case it produces an associative algebra, this does not happen if one considers, for instance, Lie or Jordan, or alternative algebras. Thus one may call a $\mathbb{Z}_2$-graded associative algebra a \textsl{superalgebra}, both terms mean the same. But in the case of Lie (or Jordan, or whatever) algebras, a Lie superalgebra will be a $\mathbb{Z}_2$-graded algebra such that $E(A)$ is a Lie algebra. Clearly a Lie superalgebra seldom will be a Lie algebra. 

By \cite[Theorem 1.3]{representability} there exists a finite dimensional $\mathbb{Z}_{2}$-graded algebra $C=C_{0}\oplus C_{1}$ such that $T_{\mathbb{Z}_{2}}(E(C))=T_{\mathbb{Z}_{2}}(A)=T_{\mathbb{Z}_{2}}(E)$. At this point, it is important to mention that in \cite{kemer.1} and \cite{kemer.2}, Kemer proved that for any associative algebra $R$ over a field of characteristic $0$, there exists a finite dimensional $\mathbb{Z}_{2}$-graded algebra $S$ such that $T(R)=T(E(S))$.

Before stating the next result, we recall a simple but important fact about tensor products of vector spaces.
\begin{Remark}
\label{remark.tensor.product}
If $V$ and $W$ are vector spaces, then by \cite[Theorem 14.5]{roman2005advanced}, we have that $v \otimes w = 0$ in $V \otimes W$ if and only if $v = 0$ or $w = 0$.   
\end{Remark}

\begin{Lemma}
\label{equalitty.T.G}
Let $R$ be a $\mathbb{Z}_{2}$-graded algebra such that $T_{\mathbb{Z}_{2}}(A)=T_{\mathbb{Z}_{2}}(R)$. Then, $R$ is a $\mathbb{Z}_{2}$-graded regular algebra with minimal regular decomposition. 
\end{Lemma}
\begin{proof} 
Since, by our standing assumption, $A$ is a $\mathbb{Z}_2$-graded regular algebra, and its regular decomposition is minimal, $T_{\mathbb{Z}_{2}}(E)=T_{\mathbb{Z}_{2}}(A)$. It follows that $R$ satisfies the same graded identities as the Grassmann algebra $E$. Since $E$ has no graded monomial identities, we conclude that $R$ likewise has no graded monomial identities. Thus, $R=R_{0}\oplus R_{1}$ satisfies the first condition of regularity.  

By \cite[Proposition 2]{giambruno2001polynomial}, $T_{\mathbb{Z}_{2}}(E)$ is generated (as a $T_{G}$-ideal) by the graded polynomials
\[
[x_{1}^{(0)},x_{2}^{(0)}],\quad [x_{1}^{(0)},x_{1}^{(1)}],\quad \text{and}\quad x_{1}^{(1)}x_{2}^{(1)}+x_{2}^{(1)}x_{1}^{(1)}.
\]
Therefore, for any $a_{0}$, $b_{0}\in R_{0}$, $a_{1}$, $b_{1}\in R_{1}$ we get
\[
a_{0}b_{0}=b_{0}a_{0},\quad a_{0}b_{1}=b_{1}a_{0},\quad \text{and}\quad a_{1}b_{1}=-b_{1}a_{1}
\]
which implies $R$ satisfies the second condition of regularity. Hence $R$ is a $\mathbb{Z}_{2}$-graded regular algebra whose regular decomposition is minimal since $M^{R}=M^{E}$.
\end{proof}
We observe here that the equality of the ideals of graded identities of two algebras implies the equality of the ideals of their ordinary ones. Therefore $T_{\mathbb{Z}_{2}}(A)=T_{\mathbb{Z}_{2}}(R)$ implies $T(A)=T(R)=T(E)$, that is the ordinary identities of $R$, and of $A$, follow from the triple commutator, namely from the identity $[[x_1,x_2], x_3]$. 
\begin{Lemma} 
\label{commutative}
The $\mathbb{Z}_{2}$-graded algebra $C$ is regular and commutative. 
\end{Lemma} 
\begin{proof}
We recall that $C$ is a finite dimensional $\mathbb{Z}_{2}$-graded algebra such that $T_{\mathbb{Z}_{2}}(E(C))=T_{\mathbb{Z}_{2}}(A)=T_{\mathbb{Z}_{2}}(E)$. 
Since $T_{\mathbb{Z}_{2}}(E(C))=T_{\mathbb{Z}_{2}}(A)$, it follows from Lemma \ref{equalitty.T.G} that $E(C)$ is a $\mathbb{Z}_{2}$-graded regular algebra with skew-symmetric bicharacter $\beta$. First we will show that $C$ is commutative.  Take $0\neq a\in E_{0}$, $0\neq e\in E_{1}$, $c_{0}$, $c_{0}'\in C_{0}$, $c_{1}\in C_{1}$, then we have
\begin{equation}
\label{e.1}
    a\otimes (c_{0}c_{0}'-c_{0}'c_{0})=(a\otimes c_{0})(1\otimes c_{0}')-(1\otimes c_{0}')(a\otimes c_{0})=0,
\end{equation}
\begin{equation}
\label{e.2}
    e\otimes (c_{0}'c_{1}-c_{1}c_{0}')=(1\otimes c_{0}')(e\otimes c_{1})-(e\otimes c_{1})(1\otimes c_{0}')=0,
\end{equation}
so, by Remark \ref{remark.tensor.product}, the equations (\ref{e.1}) and (\ref{e.2}) imply respectively that $c_{0}c_{0}'=c_{0}'c_{0}$ and $c_{0}'c_{1}=c_{1}c_{0}'$. Hence, by arbitrariness of $c_{0}$, $c_{0}'\in C_{0}$, $c_{1}\in C_{1}$, we have $C_{0}$ is commutative and commutes with $C_{1}$. Let us show that $C_{1}$ is commutative.  Let $c_{1}$, $c_{1}'\in C_{1}$, $0\neq e\in E_{1}$ and take $0\neq e'\in E_{1}$ such that $ee'\neq 0$, then we have
\begin{align*}
    (ee'\otimes (c_{1}c_{1}'-c_{1}'c_{1})) &= (e\otimes c_{1})(e'\otimes c_{1}')-(e\otimes c_{1}')(e'\otimes c_{1})\\
    &= (e\otimes c_{1})(e'\otimes c_{1}')-(ee'\otimes c_{1}'c_{1})\\
    &= (e\otimes c_{1})(e'\otimes c_{1}')+(e'e\otimes c_{1}'c_{1})\\
    &=  (e\otimes c_{1})(e'\otimes c_{1}')+(e'\otimes c_{1}')(e\otimes c_{1})=0.
    \end{align*}
  Once again, by Remark \ref{remark.tensor.product}, it follows that $c_{1}c_{1}' = c_{1}'c_{1}$. Therefore, since $c_{1}$, $c_{1}' \in C_{1}$ were arbitrary, we conclude that $C_{1}$ is commutative. Hence, $C$ is commutative.
    
    Now, suppose $C$ is not regular. Then there exists a $k$-tuple $(u_{1},\ldots, u_{k})\in (\mathbb{Z}_{2})^{k}$ such that $c_{1}\cdots c_{k}=0$, for every $c_{i}\in C_{u_{i}}$. Therefore, for every $a_{i}\in E_{u_{i}}$, $1\leq i\leq k$, we obtain
    \[
    (a_{1}\otimes c_{1})\cdots (a_{k}\otimes c_{k})=a_{1}\cdots a_{k}\otimes c_{1}\cdots c_{k}=0
    \]
which is a contradiction because $E(C)$ is regular. Consequently, $C$ is a commutative $\mathbb{Z}_{2}$-graded algebra. 
    \end{proof}

Before proceeding, let us recall the classification of the finite dimensional $\mathbb{Z}_{2}$-graded simple algebras obtained by C. T. C. Wall in his 1963 paper \cite{wall}.

\begin{Theorem}
\label{classification}
\cite{wall} Let $A$ be a finite dimensional $\mathbb{Z}_{2}$-graded simple algebra. Then, $A$ is isomorphic to one of the $\mathbb{Z}_{2}$-graded algebras:
\begin{enumerate}
    \item $A_{1}:=M_{n}(K)$, with the trivial grading.
    \item $A_{2}:=M_{k,l}(K)=\Bigg\{\begin{pmatrix}
	P & Q\\
	R & S
	\end{pmatrix}\mid P\in M_{k}(K),Q\in M_{k,l}(K),R\in M_{l,k}(K), S\in M_{l}(K)\Bigg\}$, for some $k\geq l>0$, where the $\mathbb{Z}_{2}$-grading is given by
    \[
    (A_{2})_{0}=\Bigg\{\begin{pmatrix}
	P & 0\\
	0 & S
	\end{pmatrix}\Bigg\}, \qquad 
	(A_{2})_{1}=\Bigg\{ \begin{pmatrix}
		0 & Q\\
		R & 0
		\end{pmatrix}\Bigg\}.
		\]
    \item $A_{3}:=M_{n}(K\oplus cK)$, with $c^{2}=1$, and  $\mathbb{Z}_{2}$-grading given by
	$(A_{3})_{0}=M_{n}(K)$, 
	$(A_{3})_{1}=cM_{n}(K)$.
    \end{enumerate}
\end{Theorem}

\begin{Lemma}
\label{aux.Lemma.1}
If $C$  is a finite dimensional $\mathbb{Z}_{2}$-graded algebra such that $T_{\mathbb{Z}_{2}}(A)=T_{\mathbb{Z}_{2}}(E(C))$, then there exist $r\geq 0$ and $p\geq 1$ such that
    \[
    C\cong \underbrace{K\oplus \cdots \oplus K}_{r}\oplus \underbrace{(K\oplus c_{r+1}K)\oplus \cdots \oplus(K\oplus c_{r+p}K)}_{p}\oplus J(C).
    \]
   where, if $W_{i}\cong K\oplus c_{i}K$, then $c_{i}^{2}=1_{W_{i}}$, the unit of $W_i$, for all $r+1\leq i\leq r+p$. 

\end{Lemma}
\begin{proof} 
Since $E(C)$ is a $\mathbb{Z}_{2}$-graded regular algebra with the same skew-symmetric bicharacter as $A$, it follows that the regular decomposition of $E(C)$ is minimal and $\operatorname{Supp}_{\mathbb{Z}_{2}}(E(C))=\mathbb{Z}_{2}$ and, in particular, $C_{1}\neq 0$. 

By the Wedderburn–Malcev theorem for finite dimensional $\mathbb{Z}_{2}$-graded algebras \cite[Theorem 3.5.4]{GZbook}, we can write $C = C' \oplus  J(C)$, a direct sum of vector spaces, where $C'$ is a direct sum of graded simple algebras, and $J(C)$ is the Jacobson radical of $C$. By Lemma \ref{commutative}, $C$ is commutative, so, by Theorem \ref{classification}, we may write
\[
C'=W_{1} \oplus \cdots \oplus W_{s}
\]
where each $W_{i}$ is graded isomorphic to either $K$ or $K \oplus c_{i}K$, with $c_{i}^{2} = 1$. 

Notice that
\[
C_{0}=(W_{1})_{0}\oplus \cdots \oplus (W_{s})_{0}\oplus J(C)_{0},\quad \text{and}\quad C_{1}=(W_{1})_{1}\oplus \cdots \oplus (W_{s})_{1}\oplus J(C)_{1}.
\]
Now, since $\dim C < \infty$, there exists $q \in \mathbb{N}$ such that $J(C)^q = 0$. Given $\ell>q$, by Lemma \ref{Grassmann.finite} there exists a graded embedding $\gamma\colon E_{\ell}\rightarrow E(C)$, where $E_{\ell}$ is the Grassmann algebra finitely generated by $e_{1}$,\dots, $e_{\ell}$.  

Given indices $1 \leq i_{1} < \cdots < i_{k}\leq \ell$, denote $f_{i_{1}, \ldots, i_{k}} := e_{i_{1}} \cdots e_{i_{k}}$. Then, for every $i\leq \ell$, we have
\begin{equation}
\label{aux.1}
    \gamma(e_{i}) = \sum_{t} v_{t,i} \otimes u_{t,i},
\end{equation}
where each $v_{t,i}$ may be assumed to be of the form $f_{i_{1}, \ldots, i_{k}}$, and the $u_{t,i} \in C_{1}$ are linearly independent elements. 

We claim that there exists a non-zero $u_{t,j}\notin J(C)$ for some $j\leq \ell$ and some $t$. 
 Otherwise, given $q<r'\leq \ell$, we will have
\[
\gamma(e_{1} \cdots e_{r'}) = \gamma(e_{1}) \cdots \gamma(e_{r'}) = \sum_{t} (\star) \otimes\Big( \underbrace{*\cdots *}_{r'\text{ elements from } J(C)}\Big) = 0.
\]
Because $\gamma$ is injective, it follows that $e_{1} \cdots e_{r'} = 0$, which is a contradiction. Now, if for every $1\leq i\leq s$, $W_{i}$ is graded isomorphic to $K$, then by Theorem \ref{classification}, we get $C_{1}=J(C)_{1}$. Thus, for every $j \in \mathbb{N}$, the elements $u_{t,j}$ appearing in (\ref{aux.1}) would necessarily belong to $J(C)$, which contradicts what we have just shown.

Consequently, there exists at least one $W_{i}$ that is graded isomorphic to $K \oplus c_{i}K$, with $c_{i}^{2} = 1_{W_{i}}$. In particular, if we set
\[
r:= \#\{1\leq i\leq s\mid W_{i}\cong K\}|,\quad \text{and}\quad p:=\#\{1\leq i\leq s\mid W_{i}\cong K\oplus c_{i}K,\quad c_{i}^{2}=1_{W_{i}}\}
\]
then, we conclude $p\geq 1$. Furthermore $r+p=s$.

Therefore, we get
 \[
    C\cong \underbrace{K\oplus \cdots \oplus K}_{r}\oplus \underbrace{(K\oplus c_{r+1}K)\oplus \cdots \oplus (K\oplus c_{r+p}K)}_{p}\oplus J(C).\qedhere
\]
\end{proof}

\begin{Lemma}
\label{aux.Lemma.2}
Let $D$ be a finite dimensional commutative $\mathbb{Z}_{2}$-graded regular algebra. Then the following statements hold:
	\begin{enumerate}
		\item[$(a)$]  $E(D)$ is a $\mathbb{Z}_{2}$-graded regular algebra with minimal regular decomposition. In particular $T_{\mathbb{Z}_{2}}(E(D))=T_{\mathbb{Z}_{2}}(E)$.

		\item[$(b)$] $D$ contains at least one graded simple subalgebra graded isomorphic to $K \oplus cK$, with $c^{2}=1$.    
	\end{enumerate}
\end{Lemma}
\begin{proof}  Consider the skew-symmetric bicharacter $\beta\colon \mathbb{Z}_{2}\times \mathbb{Z}_{2}\to K^{\ast}$ given by $\beta(0,0)=\beta(0,1)=\beta(1,0)=1$ and $\beta(1,1)=-1$. Given $a$, $a'\in E_{0}$, $e$, $e'\in E_{1}$, $d_{0}$, $d_{0}'\in D_{0}$, and $d_{1}$, $d_{1}'\in D_{1}$, then since $E$ is a $\mathbb{Z}_{2}$-graded regular algebra with skew-symmetric bicharacter $\beta$ and $D$ is commutative, we get
\[
(a\otimes d_{0})(a'\otimes d_{0}')=aa'\otimes d_{0}d_{0}'=a'a\otimes d_{0}'d_{0}=(a'\otimes d_{0}')(a\otimes d_{0}),
\]
\[
(a\otimes d_{0})(e\otimes d_{1})=ae\otimes d_{0}d_{1}=ea\otimes d_{1}d_{0}=(e\otimes d_{1})(a\otimes d_{0}),
\]
\[
(e\otimes d_{1})(e'\otimes d_{1}')=ee'\otimes d_{1}d_{1}'=-e'e\otimes d_{1}'d_{1}=-(e'\otimes d_{1}')(e\otimes d_{1}).
\]
The above calculations show us that $E(D)$ is a graded $\beta$-commutative algebra. Finally, let $N\in \mathbb{N}$ and $(u_{1},\ldots, u_{N})\in \mathbb{Z}_{2}^{N}$. By the regularity of $D$ there exist $d_{1}$, \dots, $d_{N}$, where $d_{i}\in D_{u_{i}}$, $1\leq i\leq N$, such that $d_{1}\cdots d_{N}\neq 0$. On the other hand, by the regularity of $E$, there exist $X_{1}$, \dots, $X_{N}$, with $X_{i}\in E_{u_{i}}$, $1\leq i\leq N$, such that $X_{1}\cdots X_{N}\neq 0$. Then, by Remark \ref{remark.tensor.product} we have
\[
(X_{1}\otimes d_{1})\cdots (X_{N}\otimes d_{N})=(X_{1}\cdots X_{N})\otimes (d_{1}\cdots d_{N})\neq 0.
\]
Therefore $E(D)$ is a $\mathbb{Z}_{2}$-graded regular algebra with skew-symmetric bicharacter $\beta$. Hence, by Lemma \ref{remark.superalgebra} 
\[
T_{\mathbb{Z}_{2}}(E(D))=T_{\mathbb{Z}_{2}}(E).
\] 
This finishes the proof of $(a)$. The proof of $(b)$ follows directly from item $(a)$ and Lemma \ref{aux.Lemma.1}.
\end{proof}

The following lemma shows that the existence of a graded simple subalgebra isomorphic to $K \oplus cK$, with $c^{2} = 1$, in a commutative $\mathbb{Z}_{2}$-graded algebra $C$, is a sufficient condition to ensure $C$ being regular.

\begin{Lemma}\label{aux.example} Let $C$ be a commutative $\mathbb{Z}_{2}$-graded algebra and suppose that $C$ contains at least one graded simple subalgebra isomorphic to $K \oplus cK$, with $c^{2}=1$. Then, $E(C)$ is a $\mathbb{Z}_{2}$-graded regular algebra with minimal regular decomposition.
\end{Lemma}

\begin{proof} Denote by $D$ the graded simple subalgebra of $C$ isomorphic to $K\oplus cK$, with $c^{2}=1$. Without loss of generality, we may assume $D = K \oplus cK$. Of course, $E(D)$ is regular because it is isomorphic to $E$. Since clearly $E(C)$ is a $\mathbb{Z}_{2}$-graded $\beta$-commutative algebra, it follows that $E(C)$ is a $\mathbb{Z}_{2}$-graded regular algebra with minimal regular decomposition  and we are done.
\end{proof}

\begin{Theorem}
\label{characterization.variety}
Let $\mathfrak{V}^{gr}=\operatorname{var}^{\mathbb{Z}_2}(E)$ be the graded variety generated by the Grassmann algebra $E$, with its natural $\mathbb{Z}_{2}$-grading. Let $A$ be a $\mathbb{Z}_{2}$-graded algebra, and let $P$ be a finite dimensional $\mathbb{Z}_{2}$-graded algebra such that $T_{\mathbb{Z}_{2}}(A) = T_{\mathbb{Z}_{2}}(E(P))$.  Then the following three conditions are equivalent:
\begin{enumerate}
    \item[(a)] $A$ generates $\mathfrak{V}^{gr}$.
    \item[(b)] $P$ is a $\mathbb{Z}_{2}$-graded regular commutative algebra.
    \item[(c)] $P$ is a commutative $\mathbb{Z}_{2}$-graded algebra, and it contains at least one graded simple subalgebra graded isomorphic to $K \oplus cK$, with $c^{2}=1$.    
\end{enumerate} 

If one (and consequently all) of the above conditions is satisfied, then $A$ is an infinite dimensional  $\mathbb{Z}_{2}$-graded regular algebra with minimal regular decomposition.

\end{Theorem}
\begin{proof}
If $A$ generates $\mathfrak{V}^{gr}$, by Lemma \ref{equalitty.T.G}, $A$ is a $\mathbb{Z}_{2}$-graded algebra with minimal regular decomposition. 
    By Lemma \ref{commutative} we get that condition $(a)$ implies $(b)$. By Lemma \ref{aux.Lemma.2} we have that $(b)$ implies $(c)$. 
      Finally, by Lemma \ref{aux.example}, it turns out that  $E(P)$ is a $\mathbb{Z}_{2}$-graded regular algebra with minimal regular decomposition. Hence by Lemma \ref{remark.superalgebra} we get $T_{\mathbb{Z}_{2}}(E(P))=T_{\mathbb{Z}_{2}}(E)$, thus $T_{\mathbb{Z}_{2}}(A)=T_{\mathbb{Z}_{2}}(E)$ which proves that $(c)$ implies $(a)$.
The last part of the statement follows directly from Lemma \ref{remark.superalgebra}.
\end{proof}

\section{\texorpdfstring{Remarks on finitely generated subalgebras of $\mathbb{Z}_2$-regular algebras}{Remarks on finitely generated subalgebras of Z2-regular algebras}}

Let $A$ be a $\mathbb{Z}_2$-graded regular algebra with a minimal regular decomposition. Given homogeneous elements $a_1$, \dots, $a_n \in A$,  we denote by $\langle a_1,\ldots,a_n\rangle$ the graded subalgebra of $A$ generated by $\{a_1,\ldots,a_n\}$. Denote by $\textbf{Fin}_A$ the set of all finitely generated graded subalgebras of $A$. This set can be partially ordered by inclusion.

The following remark is basic, yet crucial for what follows.

\begin{Remark}  Let $B = \langle a_1, \dots, a_n \rangle\in \mathbf{Fin}_A$.
\label{Remark.write}
\begin{enumerate} 
\item[(1)] Then $B_0$ is spanned, as a vector space over $K$, by monomials of the form $a_{j_1} \cdots a_{j_m}$ such that $\deg_{\mathbb{Z}_2}(a_{j_1} \cdots a_{j_m}) = 0$. Similarly, $B_1$ is spanned by monomials of homogeneous degree $1$.
\item[(2)] If $\gamma a_{j}+\delta a_{i}\in \langle a_{1},\ldots, a_{i-1},a_{i+1},\ldots,a_{j-1},a_{j+1}\ldots, a_{n}\rangle$, for some $1\leq i,j\leq n$ and $\gamma$, $\delta\in K$ with either $\gamma\neq 0$ or $\delta\neq 0$, then clearly
\[
B= \langle a_{1},\ldots, a_{i-1}, \gamma a_{j}+\delta a_{i},a_{i+1},\ldots,a_{j-1},a_{j},a_{j+1},\ldots, a_{n}\rangle,
\]
 hence  $B= \langle a_{1}, \ldots, a_{i-1},a_{i+1},\ldots,a_{j-1},a_{j},a_{j+1},\ldots, a_{n}\rangle$.
    
\end{enumerate}
 
\end{Remark}

\begin{Theorem}
\label{classification.f.g} 
Given $B \in \mathbf{Fin}_A$ with $B = \langle a_{1},\ldots, a_{n}\rangle$, let $q$ denote the number of generators $a_i$ of homogeneous degree $0$. Then, there exists a finitely generated commutative algebra $C$ (with trivial grading) such that $B$ can be viewed as a right $C$-graded module, $B = LC$, where $L$ is a graded subalgebra of $A$ graded isomorphic to one of the four types of $\mathbb{Z}_{2}$-graded algebras:
\begin{itemize}
    \item (Type I):   $L\cong K$.
    
    \item (Type II): $L \cong E_{n-q}$.
    
    \item (Type III): $L \cong (E_{l}\oplus T_{n-q-l})$, where $T_{n-q-l}= (E_{1}^{+})^{\oplus n-q-l}$ with $l\geq 1$.
    
    \item (Type IV):  There exist integers $r \geq 0$ and $l \geq 2$ such that
    \[
    L\cong E_{l} \oplus (E_{2}^{+})^{\oplus r} \oplus D_{r},
    \]
    where $D_{r}$ is a vector space such that $dd' \in E_{l} \oplus (E_{2}^{+})^{\oplus r}$, for all $d$, $d' \in D_{r}$. Moreover, 
    \[
    \dim D_{r} = (n-q) - l - 2r.
    \]
\end{itemize}
\end{Theorem}

\begin{proof}
Let $B\in \textbf{Fin}_A$ with $B=\langle a_{1},\ldots,a_{n}\rangle$. Assume without loss of generality $a_{1}$, \dots, $a_{q}\in A_{0}$, for some $q\geq 1$, and $a_{q+1}$, \dots, $a_{n}\in A_{1}$. Moreover, by Remark \ref{Remark.write}, we can assume for all $1\leq i,j\leq n$ 
\[
\operatorname{Span}_{K}\{a_{i},a_{j}\}\cap  \langle a_{1},\ldots, a_{i-1},a_{i+1},\ldots,a_{j-1},a_{j+1}\ldots, a_{n}\rangle=0. 
\]
Notice that for any $r \in \mathbb{N}$ and $i_{1}$, \dots, $i_{r} \in \{1, \ldots, n\}$, there exists a monomial $M(a_{1}, \ldots, a_{q})$ in $K[a_{1}, \ldots, a_{q}]$ and indices $q+1 \leq \lambda_{1} < \lambda_{2} < \cdots < \lambda_{s} \leq n$ such that
\begin{equation}
\label{key.main}
    a_{i_{1}}\cdots a_{i_{r}}= (a_{\lambda_{1}}\cdots a_{\lambda_{s}})M(a_{1},\ldots,a_{q}).
\end{equation}
Therefore, it follows that $B=LC$, where $C:=K[a_{1},\ldots,a_{q}]$ and $L:=\langle a_{q+1},\ldots, a_{n}\rangle$.

We have the following cases:

\begin{enumerate}
    \item Suppose $q=n$. In this case, $a_{1}$, \dots, $a_{n}\in A_{0}$. Hence $B$ is isomorphic to $C:=K[a_{1},\ldots, a_{n}]$ with the trivial grading and $L\cong K$. 
    \item Suppose $q<n$. Set
    \[
    l:=\max\Big\{k\mid a_{j_{1}}\cdots a_{j_{k}}\neq 0,\quad j_{1},\ldots, j_{k}\in \{q+1,\ldots, n\}\Big\},
    \]
     then consider $k_{1}$, \dots, $k_{l}\in \{q+1,\ldots,n\}$ such that $a_{k_{1}}\cdots a_{k_{l}}\neq 0$.  Without loss of generality we can assume $k_{1}<\cdots<k_{l}$ and we will write $c_{1}:=a_{k_{1}}$, \dots, $c_{l}:=a_{k_{l}}$.   
    \begin{enumerate}
        \item[(i)] Suppose $l=n-q$.  By the same arguments used in Section 2, $c_{1}$, \dots,  $c_{n-q}$ are linearly independent and the graded subalgebra generated by $c_{1}$, \dots, $c_{n-q}$ is graded isomorphic to $E_{n-q}$, the Grassmann algebra on $n-q$ generators. It follows that $L$ is graded isomorphic to $E_{n-q}$. 
        \item[(ii)] If $l=1$, then, $L=(K1\oplus Ka_{q+1})\oplus Ka_{q+2}\oplus \cdots \oplus Ka_{n}$ and $a_{i}a_{j}=0$, for all $q+1\leq i\leq n$. In other words, $L$ is graded isomorphic to $E_{1}\oplus T_{n-q-1}$, where $T_{n-q-1}:=(E_{1}^{+})^{\oplus n-q-1}$, and the multiplication between elements from different summands in $T_{n-q-1}$ is zero.
       
        \item[(iii)] Suppose $1<l<n-q$ and $l=2$. In this case $a_{u}a_{v}a_{w}=0$ for any $u$, $v$, $w\in \{q+1,\ldots, n\}$, and $\Lambda_{2}:=\langle a_{k_{1}},a_{k_{2}}\rangle$ is graded isomorphic to $E_{2}$.  
        Consider
        \[
    l':=\max\Big\{s\mid a_{r_{1}}\cdots a_{r_{s}}\neq 0,\quad r_{1},\ldots, r_{s}\in \{q+1,\ldots, n\}\setminus\{k_{1},k_{2}\}\Big\};
        \]
        obviously, $l'\in \{1,2\}$.  If $l'=1$, then, the remaining elements of $\{a_{q+1}, \ldots, a_{n}\} \setminus \{a_{k_{1}}, a_{k_{2}}\}$ are elements that annihilate each other, that is, they generate $T_{n-q-2}$, and, by  assumption, we have $T_{n-q-2} \cap \Lambda_{2}= 0$.  In this case we obtain 
        \[
        L\cong E_{2}\oplus T_{n-q-2}.
        \] 
        Suppose $l'=2$. Let $\mathbf{I}:=(\{q+1,\ldots,{n}\}\setminus \{k_{1},k_{2}\})$. If for every $(u,v)\in \mathbf{I}^{2}$ with $u<v$ we have either $a_{u}a_{v}=0$ or $a_{u}a_{v}\in \Lambda_{2}$, then the vector space $D_{1}:=\bigoplus_{i\in \mathbf{I}} Ka_{i}$ satisfies the following property: for all $d,d'\in D_{1}$, one has $dd'\in \Lambda_{2}$. In this case
        \[
        L\cong E_{2}\oplus D_{1}
        \]
        and $\dim D_{1}=(n-q)-2$. Now, suppose that there exists a pair $(k_{1}',k_{2}')\in \mathbf{I}^{2}$, $k_{1}'<k_{2}'$, such that $0\neq a_{k_{1}'}a_{k_{2}'}\notin \Lambda_{2}$. In this case, notice that if for some $\gamma_{1}$, $\gamma_{2}$, $\gamma_{3}\in K$, $\gamma_{1}a_{k_{1}'}+\gamma_{2}a_{k_{2}'}+\gamma_{3}a_{k_{1}'}a_{k_{2}'}\in \Lambda_{2}$, then  since $\Lambda_{2}$ is a graded subalgebra, $\deg(a_{i})=1$, for $i\in \{k_{1}',k_{2}'\}$ and $\deg(a_{k_{1}'}a_{k_{2}'})=0$, we would have $a_{k_{1}'}a_{k_{2}'}\in \Lambda_{2}$. Hence $\gamma_{3}=0$ and $\gamma_{1}a_{k_{1}'}+\gamma_{2}a_{k_{2}'}\in \Lambda_{2}$. Therefore
        \[
        \gamma_{1}a_{k_{1}'}+\gamma_{2}a_{k_{2}'}\in \langle a_{1},\ldots, a_{k_{1}'-1},a_{k_{1}'+1},\ldots, a_{k_{2}'-1},a_{k_{2}'+1},\ldots, a_{n}\rangle,
        \]
        which implies $\gamma_{1}=\gamma_{2}=0$. It follows that the subalgebra $U_{2}':=\operatorname{Span}_{K}\{a_{k_{1}'},a_{k_{2}'},a_{k_{1}'}a_{k_{2}'}\}$, which is graded isomorphic to $E_{2}^{+}$, is such that
        \[
        U_{2}'\cap \Lambda_{2}=0.
        \]
         We repeat the procedure described above: we find a pair of elements $(u,v)\in (\mathbf{I}\setminus\{k_{1}',k_{2}'\})^{2}$ such that $u<v$ and $0\neq a_{u}a_{v}\notin (\Lambda_{2}\oplus U_{2}')$. If no such pair exists, we construct a vector space $D'$ such that for all $f$, $f'\in D'$, one has $ff'\in \Lambda_{2}\oplus U_{2}'$. Eventually, we get that there exists $r\geq 0$ such that 
\[
L\cong E_{2}\oplus (E_{2}^{+})^{\oplus r}\oplus D_{r+1},
\]
where $D_{r+1}$ has dimension $(n-q)-2(r+1)$ and has the property that for all $d_{1}$, $d_{2}\in D_{r+1}$, $d_{1}d_{2}\in E_{2}\oplus (E_{2}^{+})^{\oplus r}$.  

\item[(iv)] Suppose $1<l<n-q$ and $l\geq 2$.  Analogously to (i) we get a graded subalgebra $\Lambda_{l}$ which is graded isomorphic to $E_{l}$.  Consider 
\[
l':=\max\Big\{s\mid a_{r_{1}}\cdots a_{r_{s}}\neq 0,\quad r_{1},\ldots, r_{s}\in \{q+1,\ldots, n\}\setminus\{k_{1},\ldots, k_{l}\}\Big\}.
\]
If $l'=1$, then $L\cong (E_{l}\oplus T_{n-q-l})$.

If $l'\geq 2$, consider the set $\mathbf{J}=\{q+1,\ldots,n\}\setminus \{k_{1},\ldots, k_{l}\}$. If for all $(u,v)\in \mathbf{J}^{2}$, with $u<v$, we have $a_{u}a_{v}\in \Lambda_{l}$, then 
\[
L\cong E_{l}\oplus \widehat{D}_{1},
\]
where $\widehat{D}_{1}$ is a vector space, with dimension $\dim \widehat{D}_{1}= (n-q)-l$ such that for all $d$, $d'\in \widehat{D}_{1}$, $dd'\in E_{l}$. Suppose there exists a pair $(u,v)\in \mathbf{J}^{2}$, $u<v$, satisfying $0\neq a_{u}a_{v}\notin E_{l}$. The graded subalgebra $W_{2}':= \operatorname{Span}_{K}\{a_{u},a_{v},a_{u}a_{v}\}$, which is graded isomorphic to $E_{2}^{+}$, satisfies $W_{2}'\cap \Lambda_{l}=0$. Then we consider
    \[
        \max\Big\{s\mid a_{r_{1}}\cdots a_{r_{s}}\neq 0,\quad r_{1},\ldots, r_{s}\in \{q+1,\ldots, n\}\setminus\{k_{1},\ldots, k_{l},u,v\}\Big\}
    \]
 and, as in (iii), there exists $r\geq 0$ such that 
\[
L\cong E_{l}\oplus (E_{2}^{+})^{\oplus r}\oplus \widehat{D}_{r}, 
\]
where $\widehat{D}_{r}$ is a vector space of dimension $(n-q)-l-2r$ such that for all $d$, $d'\in \widehat{D}_{r}$, we have $dd'\in E_{l}\oplus (E_{2}^{+})^{\oplus r}$. \qedhere
    \end{enumerate}
\end{enumerate}
\end{proof}

The next result establishes a connection with Proposition \ref{family.regular} on regular gradings constructed from $m$-regular gradings via direct limits.

\begin{Corollary}
The $\mathbb{Z}_{2}$-graded regular algebra $A$ is the direct limit of $\mathbb{Z}_{2}$-graded algebras, each of which is $k$-regular for some $k \in \mathbb{N}$ (where $k$ may vary for different algebras).
\end{Corollary}

\begin{proof}
It is easy to see that $\mathbf{Fin}_{A}$ is a directed set; that is, for all $B$, $B' \in \mathbf{Fin}_{A}$ there exists $B'' \in \mathbf{Fin}_{A}$ such that $B \subseteq B''$ and $B' \subseteq B''$. Hence, it follows that  $A = \varinjlim_{\mathbf{Fin}_{A}} B$. Given $B \in \mathbf{Fin}_{A}$, by Theorem \ref{classification.f.g} we can write 
\[
B = LC = (L)_{0}C \oplus (L)_{1}C,
\]
where $L$ is of Type I, II, III or IV. It is straightforward to verify that every such $B$ is a $\mathbb{Z}_{2}$-graded $k$-regular grading, for some $k \in \mathbb{N}$.
\end{proof}

\begin{Remark}
 In light of Theorem \ref{characterization.variety}, it follows that if $W$ generates the variety $\mathfrak{V}^{gr}$, then its finitely generated graded subalgebras are of the form $RC$, where $C$ is a finitely generated commutative subalgebra of $W$, and $R$ is of Type I, II, III or IV.

\end{Remark}

\section*{Statements and Declarations}
The authors have no relevant financial or non-financial interests to disclose.

No data are associated with this research.

\section*{Acknowledgments}
We thank the referee for his/her careful reading of the first version of the paper. The referee's comments have been incorporated into the manuscript, resulting in a significant improvement of the presentation. We point out that Example~\ref{example_referee} was included in the paper as a consequence of one of the referee's remarks.

\end{document}